\documentclass[12pt]{amsart}
\usepackage{graphicx}
\usepackage{amscd}
\usepackage[top=1in,bottom=1in,left=1.15in,right=1.15in]{geometry}
\usepackage{amsmath,amssymb,amsfonts}
\usepackage{amsthm}
\usepackage{color}
\usepackage{hyperref}
\usepackage{tikz-cd}
\usepackage{epstopdf}
\usetikzlibrary{matrix}
\usepackage{verbatim}
\usepackage{array}
\usepackage{mathtools}
\usepackage{enumitem}
\newtheorem{theo}{Theorem}[section]
\newtheorem{prop}[theo]{Proposition}
\newtheorem{lemma}[theo]{Lemma}

\newtheorem{cor}[theo]{Corollary}

\theoremstyle{definition}
\newtheorem{definition}[theo]{Definition}
\newtheorem{example}[theo]{Example}
\newcommand{\cal}{\mathcal}

\allowdisplaybreaks

\def\diaCrossP{\unitlength.08em
  \begin{minipage}{15\unitlength}
    \begin{picture}(15,15)
      \put(0,0){\vector(1,1){15}}
      \qbezier(15,0)(15,0)(10,5)
      \qbezier(5,10)(0,15)(0,15)
      \put(0,15){\vector(-1,1){0}}
    \end{picture}
  \end{minipage}
}

\def\diaCrossN{\unitlength.08em
  \begin{minipage}{15\unitlength}
    \begin{picture}(15,15)
      \put(15,0){\vector(-1,1){15}}
      \qbezier(0,0)(0,0)(5,5)
      \qbezier(10,10)(15,15)(15,15)
      \put(15,15){\vector(1,1){0}}
    \end{picture}
  \end{minipage}
}

\def\diaCross{\unitlength.08em
  \begin{minipage}{15\unitlength}
    \begin{picture}(15,15)
      \put(0,0){\vector(1,1){15}}
      \put(15,0){\vector(-1,1){15}}
    \end{picture}
  \end{minipage}
}

\def\diaCircle{\unitlength.1em
  \begin{minipage}{15\unitlength}
    \begin{picture}(15,15)
      \put(7.5,10){\circle{8}}
      \put(7.5,0){\vector(0,1){5}}
       \put(7.5,5){\line(0,1){5.5}}
    \end{picture}
  \end{minipage}
}

\usetikzlibrary{decorations.markings}

\tikzset{->-/.style={decoration={
  markings,
  mark=at position .5 with {\arrow{>}}},postaction={decorate}}}

\tikzset{-<-/.style={decoration={
  markings,
  mark=at position .5 with {\arrow{<}}},postaction={decorate}}}

\begin{document}
\title{Alexander polynomial and spanning trees}
\author{Yuanyuan Bao \and Zhongtao Wu}
\address{
Graduate School of Mathematical Sciences, University of Tokyo, 3-8-1 Komaba, Tokyo 153-8914, Japan
}
\email{bao@ms.u-tokyo.ac.jp}

 \address{
Department of Mathematics, The Chinese University of Hong Kong, Shatin, Hong Kong
}
\email{ztwu@math.cuhk.edu.hk}

\begin{abstract}

Inspired by the combinatorial constructions in earlier work of the authors that generalized the classical Alexander polynomial to a large class of spatial graphs with a balanced weight on edges, we show that the value of the Alexander polynomial evaluated at $t=1$ gives the weighted number of the spanning trees of the graph.  

\end{abstract}

\maketitle

\section{Introduction}

In \cite{BW}, we studied an Alexander polynomial $\Delta_{(G,c)}(t)$ for a certain class of spatial graphs in the 3-sphere $S^3$. Having a standard definition in terms of abelian covers of graph complement \cite[Section 5]{bao}, the invariant is foremost a topological invariant that naturally generalizes the classical Alexander polynomial for knots and links.  On the other hand, the equivalent definitions in terms of Kauffman states and MOY calculus discovered by the authors reveal several interesting combinatorial flavour of the invariant.  In particular, it is shown that the value of the Alexander polynomial evaluated at $t=1$ is unchanged under crossing changes of the graph diagrams.  Consequently, for a spatial graph $(G, c)$, $\Delta_{(g,c)} :=\Delta_{(G,c)}(1)$ is an intrinsic invariant of the underlying abstract graph $(g, c)$ of $(G, c)$.


In this paper, we go one step further and relate the invariant $\Delta_{(g,c)}$ with a certain count of spanning trees of the graph.  In order to state the main result precisely, we introduce a few notations and terms first.

\begin{definition}
Given a vertex $r$ in a connected directed graph $\Gamma$, an {\it oriented spanning tree of $\Gamma$ rooted at $r$} is a spanning subgraph $T$ that satisfies the following $3$ conditions:
\begin{enumerate}
\item Every vertex $v\neq r$ has in-degree $1$.
\item The root $r$ has in-degree $0$.
\item $T$ has no oriented cycle.
\end{enumerate}

\end{definition}

Denote  $\cal{T}_r(\Gamma)$ the set of all oriented spanning trees of $\Gamma$ rooted at $r$.  One can then count the number of such spanning trees.  If there is in addition a weight function $w: E\rightarrow \mathbb{Z}$ on the edge set, we can count instead the weighted number of spanning trees.

\begin{definition}\label{Def:weightednumber}
Define the weight of each spanning tree $T$ by
\begin{equation}\label{treeweight}
w(T) := \prod_{e\in E(T)} w(e),
\end{equation}
where $E(T)$ is the edge set of $T$. Then,  the {\it weighted number of spanning trees rooted at $r$} is:
\begin{equation}\label{weightednumber}
N(\Gamma, w, r):=\sum_{T\in \cal{T}_r(\Gamma)} w(T).
\end{equation}
\end{definition}

In this paper, we will be mostly interested in weight functions satisfying a certain balanced property.
We review the related definitions below and refer the reader to \cite[Definition 2.1]{BW} for more details.


\begin{definition}\label{Def:moygraph}
\rm
\begin{enumerate}

\item An \emph{abstract MOY graph} is a directed graph that equipped with \emph{a positive balanced weight/coloring} $c: E\to \mathbb{N}$ such that for each vertex $v$,
\begin{equation}\label{balancedcoloring}
\sum_{\text{$e$: pointing into $v$}} c(e)=\sum_{\text{$e$: pointing out of $v$}} c(e).
\end{equation}

\item An {\it MOY graph diagram} in $\mathbb{R}^2$ is an immersion of an abstract MOY graph into $\mathbb{R}^2$, with crossing information and a \emph{transverse orientation}: through each vertex $v$, there is a straight line $L_v$ that separates the edges entering $v$ and the edges leaving $v$.

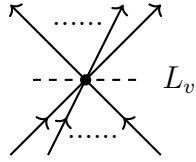
\begin{figure}[h!]
\begin{tikzpicture}[baseline=-0.65ex, thick, scale=1]
\draw (-1, -1.75) [->-] to (0, -0.75);
\draw (-0.5, -1.75) [->-] to (0, -0.75);
\draw (0, -0.75) [->] to (1, 0.25);
\draw (1, -1.75) [->-] to (0, -0.75);
\draw (0, -0.75) [->] to (-1, 0.25);
\draw (0, -0.75) [->] to (0.5, 0.25);
\draw (0, -0.75) node[circle,fill,inner sep=1.5pt]{};
\draw [dashed] (-0.7, -0.75)--(0.7, -0.75);
\draw (1.25, -0.75) node{$L_v$};
\draw (0.1, -1.5) node{$......$};
\draw (-0.1, 0) node{$......$};
\end{tikzpicture}
\label{fig:e3}
\end{figure}
\end{enumerate}
\end{definition}

An MOY graph $(G,c)$ is an equivalence class of MOY graph diagrams of $(g, c)$ under a certain topological equivalence relation (a.k.a. the {\it Reidemeister moves}). The Alexander polynomial $\Delta_{(G,c)}(t)$ is defined using an MOY graph diagram and proved to be a topological invariant for the equivalence class $(G,c)$.  


\medskip
\noindent {\bf Convention.} Throughout this paper, we only study connected graphs. As notational convention, we use $\Gamma, w$, and $w(T)$ to denote a general directed graph, a weight on $\Gamma$, and the weight of a spanning tree $T$ of $\Gamma$, respectively. In contrast, we reserve the letters $g, c$ and $c(T)$ for an abstract MOY graph, its balanced weight/coloring, and the weight of a spanning tree $T$ of $g$, respectively.\\

With the balanced property on the weight function $c$, one can show that the weighted number of spanning trees of a given abstract MOY graph $(g, c)$ is in fact independent of the choice of the root $r$ (Proposition \ref{BalancedWeightCount}). Thus we denote this number by $N(g, c)$, and our main theorem identifies it with the value $\Delta_{(g,c)}$.

\begin{theo}\label{MainTheorem}
For an abstract MOY graph $(g,c)$, we have
$$\Delta_{(g,c)}=N(g,c).$$

\end{theo}

As a corollary, we can establish the non-vanishing property for the Alexander polynomial  $\Delta_{(G,c)}(t)$ as a consequence of the existence of spanning trees, thus generalizing an earlier result of the authors \cite[Theorem 5.6]{BW}, which treated the case that $G$ is plane.

\begin{cor}\label{Nonvanishing}
Suppose $G$ is a connected MOY graph with a positive balanced weight $c$.  Then $\Delta_{(G,c)}(1)>0$.  In particular, this implies $\Delta_{(G,c)}(t)\neq 0$.
\end{cor}


\vspace{3mm}
\noindent{\bf Acknowledgements.}
We would like to thank Xian'an Jin for helpful discussions. The first named author is partially supported by JSPS KAKENHI Grant Number JP20K14304.  The second named author is partially supported by grant from the Research Grants Council of Hong Kong Special
Administrative Region, China (Project No. 14309017 and 14301819).

\section{Matrix tree theorem}

Kirchhoff's matrix tree theorem is a classical result that allows one to determine the number of spanning trees by simply computing the determinant of an appropriate matrix associated to the graph.  In this section, we review the theorem in the weighted directed graph setting.  As an application, we prove the independence of the weighted number of spanning trees on the choice of root for balanced weight.

\begin{definition}\label{Def:Laplacian}
Suppose $(\Gamma, w)$ is a weighted directed graph with vertex set $V=\{v_1, v_2, \cdots, v_n\}$.    The $n\times n$ {\it Laplacian matrix}  $L$ is given by
$$L_{ij}=\begin{cases}
-a_{ij} &  \text{if } i\neq j \\

\sum\limits_{k=1}^n a_{kj} & \text{if } i=j

\end{cases}$$
where
$$a_{ij}=\begin{cases}
\sum\limits_{\{e \,|\, e \text{ is an edge from } v_i \text{ to } v_j\}}w(e) & \text{if } i\neq j \\
0 & \text{if } i=j
\end{cases}.
$$

\end{definition}

Fix a vertex $v_r$ in $\Gamma$, and let $L_r$ be the Laplacian matrix of $\Gamma$ with the $r^{th}$ row and column removed.  The matrix tree theorem asserts:

\begin{theo}[Matrix Tree Theorem]\label{MatrixTreeTheorem}

Let $(\Gamma, w)$ be a weighted directed graph.  Then $$\det(L_r)=N(\Gamma, w, v_r),$$
where the right hand side is the weighted number of the oriented spanning trees rooted at $v_r$.
\end{theo}

A proof of the above theorem can be found, for example, in \cite{Ch}\cite{Tutte}.

In general, the weighted number of oriented spanning trees with different roots are not necessarily the same.  Nonetheless, for the most relevant case to our paper, namely, a balanced weight/coloring (Definition \ref{Def:moygraph}), $N(\Gamma, w, v_r)$ is independent of the choice of root.

\begin{prop} \label{BalancedWeightCount}
Suppose $(\Gamma, w)$ is a directed graph with a balanced weight.  We have $$N(\Gamma, w, v_i)=N(\Gamma, w, v_j)$$ for all $v_i, v_j \in V$.  In other words,
the weighted number of oriented spanning trees is independent of the choice of the root $v_r$.

\end{prop}

\begin{proof}
Recall that a balanced weight means $\sum\limits_{\text{$e$: pointing into $v$}}  w(e)=\sum\limits_{\text{$e$: pointing out of $v$}} w(e)$ for each vertex $v$.  In terms of the Laplacian matrix $L$ in Definition \ref{Def:Laplacian}, this is equivalent to the identity $\sum\limits_{k=1}^n a_{kj}=\sum\limits_{k=1}^n a_{jk}$ for all $j$; so $L$ has the property that every row and every column sums up to $0$. From $\sum\limits_j L_{ij}=0$, we can readily show that the cofactors of the elements of any particular row of $L$ are all equal.  From $\sum\limits_i L_{ij}=0$, we can likewise deduce that the cofactors of the elements of any particular column of $L$ are all equal.  Hence, all cofactors of $L$ are equal, and the statement follows from Theorem \ref{MatrixTreeTheorem}.


\end{proof}

\begin{example}\label{ExampleSpanningTree}
We consider the directed graph $\Gamma$ with a balanced weight $w$ indicated by numbers drawing near the edges, which is Fig. \ref{Fig:gamma}.
\begin{figure}[h!]
\begin{tikzpicture}[baseline=-0.65ex, thick, scale=1.5]
\draw (0, 2) node[circle,fill,inner sep=1.5pt]{};
\draw (1, 0) node[circle,fill,inner sep=1.5pt]{};
\draw (-1, 0) node[circle,fill,inner sep=1.5pt]{};

\draw (1, 0) [->-] to (0, 2);
\draw (-1, 0) [->-] to (0, 2);
\draw (-1, 0) [->-] to (1, 0);

\draw (0,2)  [->-] to [out=30,in=45] (1,0);
\draw (0,2)  [->-] to [out=150,in=135] (-1,0);
\draw (0, 2.2) node {$v_1$};
\draw (-1.2, -0.2) node {$v_2$};
\draw (1.3, -0.2) node {$v_3$};

\draw (0, -0.3) node {$k$};
\draw (0.1, 1) node {$i+k$};
\draw (-0.7, 1) node {$j$};
\draw (1.3, 1.2) node {$i$};
\draw (-1.5, 1.2) node {$j+k$};

\end{tikzpicture}
\caption{Graph $\Gamma$.}
\label{Fig:gamma}
\end{figure}
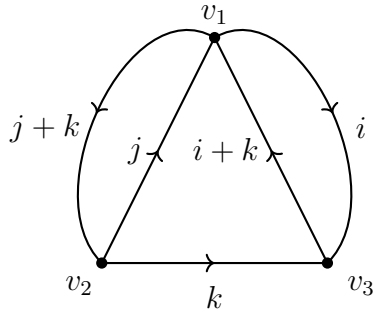


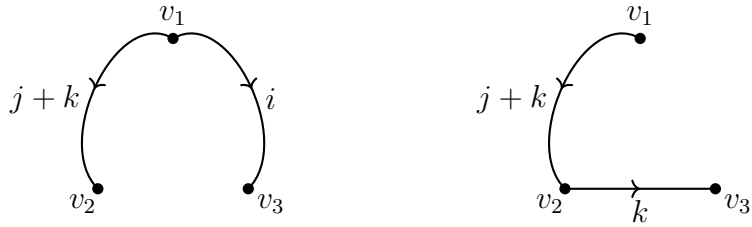
\begin{figure}[h!]
\begin{tikzpicture}[baseline=-0.65ex, thick, scale=1]
\draw (0, 2) node[circle,fill,inner sep=1.5pt]{};
\draw (1, 0) node[circle,fill,inner sep=1.5pt]{};
\draw (-1, 0) node[circle,fill,inner sep=1.5pt]{};

\draw (0,2)  [->-] to [out=30,in=45] (1,0);
\draw (0,2)  [->-] to [out=150,in=135] (-1,0);
\draw (0, 2.3) node {$v_1$};
\draw (-1.2, -0.2) node {$v_2$};
\draw (1.3, -0.2) node {$v_3$};

\draw (1.3, 1.2) node {$i$};
\draw (-1.7, 1.2) node {$j+k$};

\end{tikzpicture}
\hspace{2cm}
\begin{tikzpicture}[baseline=-0.65ex, thick, scale=1]
\draw (0, 2) node[circle,fill,inner sep=1.5pt]{};
\draw (1, 0) node[circle,fill,inner sep=1.5pt]{};
\draw (-1, 0) node[circle,fill,inner sep=1.5pt]{};

\draw (0,2)  [->-] to [out=150,in=135] (-1,0);

\draw (-1, 0) [->-] to (1, 0);
\draw (0, 2.3) node {$v_1$};
\draw (-1.2, -0.2) node {$v_2$};
\draw (1.3, -0.2) node {$v_3$};

\draw (0, -0.3) node {$k$};
\draw (-1.7, 1.2) node {$j+k$};
\end{tikzpicture}
\caption{There are $2$ oriented spanning trees rooted at $v_1$, so $N(\Gamma, w, v_1)=i(j+k)+k(j+k)=(i+k)(j+k)$.}
\label{Fig: root1}
\end{figure}

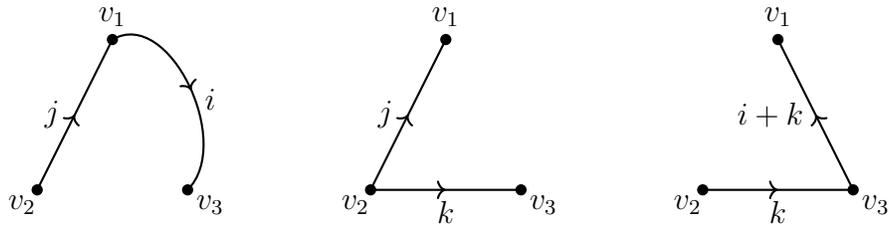
\begin{figure}[h!]
\begin{tikzpicture}[baseline=-0.65ex, thick, scale=1]
\draw (0, 2) node[circle,fill,inner sep=1.5pt]{};
\draw (1, 0) node[circle,fill,inner sep=1.5pt]{};
\draw (-1, 0) node[circle,fill,inner sep=1.5pt]{};

\draw (-1, 0) [->-] to (0, 2);

\draw (0,2)  [->-] to [out=30,in=45] (1,0);

\draw (0, 2.3) node {$v_1$};
\draw (-1.2, -0.2) node {$v_2$};
\draw (1.3, -0.2) node {$v_3$};

\draw (1.3, 1.2) node {$i$};
\draw (-0.8, 1) node {$j$};

\end{tikzpicture}
\hspace{1cm}
\begin{tikzpicture}[baseline=-0.65ex, thick, scale=1]
\draw (0, 2) node[circle,fill,inner sep=1.5pt]{};
\draw (1, 0) node[circle,fill,inner sep=1.5pt]{};
\draw (-1, 0) node[circle,fill,inner sep=1.5pt]{};

\draw (-1, 0) [->-] to (0, 2);

\draw (-1, 0) [->-] to (1, 0);
\draw (0, 2.3) node {$v_1$};
\draw (-1.2, -0.2) node {$v_2$};
\draw (1.3, -0.2) node {$v_3$};

\draw (-0.8, 1) node {$j$};
\draw (0, -0.3) node {$k$};

\end{tikzpicture}
\hspace{1cm}
\begin{tikzpicture}[baseline=-0.65ex, thick, scale=1]
\draw (0, 2) node[circle,fill,inner sep=1.5pt]{};
\draw (1, 0) node[circle,fill,inner sep=1.5pt]{};
\draw (-1, 0) node[circle,fill,inner sep=1.5pt]{};

\draw (1, 0) [->-] to (0, 2);
\draw (-1, 0) [->-] to (1, 0);
\draw (0, 2.3) node {$v_1$};
\draw (-1.2, -0.2) node {$v_2$};
\draw (1.3, -0.2) node {$v_3$};

\draw (0, -0.3) node {$k$};
\draw (-0.1, 1) node {$i+k$};
\end{tikzpicture}
\caption{There are $3$ oriented spanning trees rooted at $v_2$, so $N(\Gamma, w, v_2)=ij+kj+k(i+k)=(i+k)(j+k)$.}
\label{Fig: root2}
\end{figure}

\begin{figure}[h!]
\begin{tikzpicture}[baseline=-0.65ex, thick, scale=1]
\draw (0, 2) node[circle,fill,inner sep=1.5pt]{};
\draw (1, 0) node[circle,fill,inner sep=1.5pt]{};
\draw (-1, 0) node[circle,fill,inner sep=1.5pt]{};

\draw (1, 0) [->-] to (0, 2);

\draw (0,2)  [->-] to [out=150,in=135] (-1,0);
\draw (0, 2.2) node {$v_1$};
\draw (-1.2, -0.2) node {$v_2$};
\draw (1.3, -0.2) node {$v_3$};

\draw (-0.1, 1) node {$i+k$};

\draw (-1.7, 1.2) node {$j+k$};

\end{tikzpicture}
\caption{There is $1$ oriented spanning tree rooted at $v_3$, so $N(\Gamma, w, v_2)=(i+k)(j+k)$.}
\label{Fig: root3}
\end{figure}
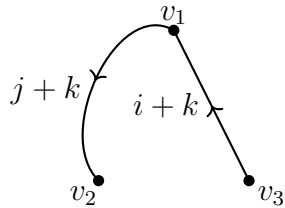

We can check that the weighted numbers of spanning trees rooted at either vertex are all equal to $(i+k)(j+k)$, as illustrated in Fig. \ref{Fig: root1}, \ref{Fig: root2} and \ref{Fig: root3}.
On the other hand, the Laplacian matrix is:

\begin{equation*}
L=\begin{pmatrix}
i+j+k & -(j+k) & -i \\
-j       & j+k    & -k \\
-(i+k) & 0     & i+k
\end{pmatrix}.
\end{equation*}

It is a straightforward calculation to see that all the matrix cofactors are also equal to $(i+k)(j+k)$.

\end{example}

\section{Spanning trees and Kauffman states}

In this section, we prove Theorem \ref{MainTheorem} for a planar MOY graph $(g,c)$, that is, there exists an MOY graph diagram $G$ of $g$ in the plane without intersections between the interior of edges. Our strategy is to express $\Delta_{(g,c)}=\Delta_{(G,c)}(1)$ using the Kauffman state sum formulation and then make an explicit bijection of Kauffman states to the oriented spanning trees in (\ref{weightednumber}) for the weighted sum.

From now on, $G$ denotes a plane MOY graph diagram in $\mathbb{R}^2$ of the graph $g$.  In \cite[Section 2]{BW}, the authors defined the Kauffman state sum for general MOY graph diagrams; we do not need the full generality here, and instead, will only focus on the simpler plane diagram case, following \cite[Section 5.3]{BW}.


Starting from the plane diagram $G$, we can obtain a {\it decorated diagram} $(G, \delta)$ by putting a base point $\delta$ on one of the edges in $G$ and drawing a circle around each vertex of $G$. Then

\begin{enumerate}
\item  $\operatorname{Cr}(G)$: denotes the set of crossings \diaCircle which are the intersection points around each vertex between the incoming edges with the circle.  Such a crossing is said to be generated by the edge. (In Example \ref{Ex1}, there are two crossings around $v_1$ generated by the edges with weights $j$ and $i+k$, respectively.) \\

\item $\operatorname{Re}(G)$: denotes the set of regions, including the {\it regular regions} of $\mathbb{R}^{2}$ separated by $G$ and the {\it circle regions} around the vertices. Note that there is exactly one circle region around each vertex. {\it Marked regions} are the regions adjacent to the base point $\delta$, and the others are called {\it unmarked regions}. (In Example \ref{Ex1}, there are $2$ marked regions and $5$ unmarked regions.)\\

\item Corners: There are $3$ corners around a crossing \diaCircle, and we call the one inside the circle region the {\it north} corner, the one on the left of the crossing the {\it west} corner and the one on the right the {\it east} corner, as illustrated below.  Note also that every corner belongs to a unique region in $\operatorname{Re}(G)$.  \\

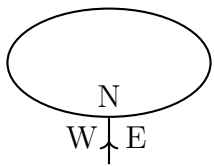
\begin{figure}[h!]
\begin{tikzpicture}[baseline=-0.65ex, thick, scale=0.9]
\draw (0, 0.5) ellipse (1.5cm and 0.8cm);
\draw (0,-1) [->-] to (0,-0.3);
\draw (-0.4, -0.6) node {W};
\draw (0.4, -0.6) node {E};
\draw (0, 0) node {N};
\end{tikzpicture}
\end{figure}

\end{enumerate}

Calculating the Euler characteristic of $\mathbb{R}^2$ using $G$ shows
$$\vert \operatorname{Re}(G) \vert = \vert \operatorname{Cr}(G) \vert+2.$$
Also, a generic base point $\delta$ is adjacent to two regions, which will be denoted by $R_u$ and $R_{v}$. Note that since we only consider a graph equipped with a positive balanced weight, $R_u$ and $R_v$ must be distinct.

\begin{definition}
\label{states}
A {\it Kauffman state} for a decorated diagram $(G, \delta)$ is a bijective map $$s: \, \operatorname{Cr}(G)\rightarrow \operatorname{Re}(G)\backslash \{R_u, R_{v}\},$$ which sends a crossing in $\operatorname{Cr}(G)$ to one of its corners. Let $S(G, \delta)$ denote the set of all Kauffman states.
\end{definition}

\begin{definition}\label{Def:KauffmanStateSum}
Suppose $(G, \delta)$ is a decorated plane diagram with $n$ crossings $C_1, C_2, \cdots, C_n$ in $\operatorname{Cr}(G)$ and $n+2$ regions $R_1, R_2, \cdots, R_{n+2}$ in $\operatorname{Re}(G)$. We assume that the base point $\delta$ is on an edge $e_1$ with weight $i_1$. 

\begin{enumerate}
\item We define a local contribution $P_{C_p}^{\triangle}(t)$ as in Fig. \ref{fig:f2}, which is a polynomial in $t$.  
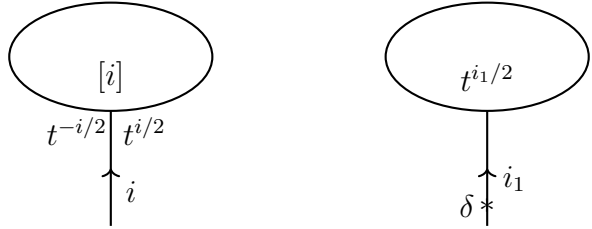
\begin{figure}[h!]
\begin{tikzpicture}[baseline=-0.65ex, thick, scale=0.9]
\draw (0, 0.5) ellipse (1.5cm and 0.8cm);
\draw (0,-2) [->-] to (0,-0.3);
\draw (-0.5, -0.6) node {$t^{-i/2}$};
\draw (0.5, -0.6) node {$t^{i/2}$};
\draw (0, 0.2) node {$[i]$};
\draw (0.3, -1.5) node {$i$};
\end{tikzpicture}
\hspace{2cm}
\begin{tikzpicture}[baseline=-0.65ex, thick, scale=0.9]
\draw (0, 0.5) ellipse (1.5cm and 0.8cm);
\draw (0,-2) [->-] to (0,-0.3);
\draw (0, 0.2) node {$t^{i_1/2}$};
\draw (0.4, -1.3) node {$i_1$};
\draw (0, -1.7) node {$*$};
\draw (-0.3, -1.7) node {$\delta$};
\end{tikzpicture}
	\caption{The local contributions $P_{C_p}^{\triangle}(t)$ for crossings generated by a generic  edge without $\delta$ (left) and the edge with $\delta$ (right), respectively.}
	\label{fig:f2}
\end{figure}

Here, $\triangle$ represents a corner around $C_p$, and 
$$[i]:= \frac{t^{i/2}-t^{-i/2}}{t^{1/2}-t^{-1/2}}=t^{\frac{i-1}{2}}+\cdots +t^{\frac{1-i}{2}}. $$

\item For each Kauffman state $s$, let $$P_s(t):= \prod_{p=1}^{n} P_{C_p}^{s(C_p)}(t).$$

\item The {\it Kauffman state sum} is defined as
\begin{equation} \label{equation: planar}
\Delta_{(G, c)}(t):=\sum_{s\in S(G, \delta)}  P_s(t).
\end{equation}

\end{enumerate}
\end{definition}


\begin{theo}[Sections 3 \& 5 of \cite{BW}]
The function $\Delta_{(G, c)}(t)$ is a topological invariant of $(G,c)$ well-defined up to $t^k$ and is independent of the choice of $\delta$.
\end{theo}

Now, we are ready to prove Theorem \ref{MainTheorem} for the planar graph case.  The key observation is the remarkable similarity in the formula of the weighted number of spanning trees in Definition \ref{Def:weightednumber} and the formula of the Kauffman state sum in Definition \ref{Def:KauffmanStateSum}.  Note that when one substitutes $t=1$ in Equation (\ref{equation: planar}), the value $\Delta_{(G, c)}(1)$ is expressed as a sum of the value $P_s(1)$ over all Kauffman states $s$, where each $P_s(1)$ is a product of local contributions $P_{C_p}^{\triangle}(1)$ as in Fig. \ref{fig:f3}.  Our goal is to describe an explicit bijection between the set of rooted spanning trees $\cal{T}_r(G)$ with the set of Kauffman states $S(G, \delta)$, and then identify the weights $c(e)$ and $c(T)$ with the local contributions $P_{C_p}^{\triangle}(1)$ and $P_s(1)$, respectively.

\begin{figure}[h!]
\begin{tikzpicture}[baseline=-0.65ex, thick, scale=0.9]
\draw (0, 0.5) ellipse (1.5cm and 0.8cm);
\draw (0,-2) [->-] to (0,-0.3);
\draw (-0.5, -0.6) node {$1$};
\draw (0.5, -0.6) node {$1$};
\draw (0, 0.2) node {$i$};
\draw (0.3, -1.5) node {$i$};
\end{tikzpicture}
\caption{The local contributions $P_{C_p}^{\triangle}(1)$ for a crossing generated by a generic edge without $\delta$; the contribution is $1$ for the edge with $\delta$, so we can ignore the term in the computation of $P_s(1)$.}
	\label{fig:f3}
\end{figure}
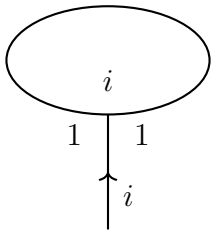

\begin{theo}\label{Correspondence}
Suppose $(G,c)$ is a plane MOY graph diagram where the base point $\delta$ is on an edge that enters the vertex $r$.  Then, there is a canonical bijective map
$$\phi: \cal{T}_r(G)   \longrightarrow   S(G, \delta) $$
so that each oriented spanning tree $T$ rooted at $r$ has a one-to-one correspondence with a certain Kauffman state $s\in S(G, \delta)$.  Moreover, the weight of each spanning tree $c(T)$ is equal to the corresponding term $P_s(1)$.   Consequently, $\Delta_{(G,c)}(1)=N(G,c).$

\end{theo}

\begin{proof}

Recall that for each spanning tree $T$ of the plane graph $G$, there is a canonical dual spanning tree $T^*$ in the dual graph $G^*$ consisting of all edges which are duals of the edges not in $T$.  We then construct the Kauffman state $s$ in the following way:

\begin{enumerate}
\item For the edge $e_1$ where the base point  $\delta$ is on, assign the crossing \diaCircle generated by $e_1$ to its north corner inside the circle region around the vertex $r$.
\item For each edge $e$ in the oriented spanning tree $T$ that enters the vertex $v$, assign the crossing \diaCircle generated by $e$ to its north corner inside the circle region around the vertex $v$.
\item For all other crossings, there is a unique way of assigning one of the east and west corners: Starting from the vertices in $G^*$ dual to the two regions $R_u$ and $R_v$, one can travel to all other vertices in $G^*$ along edges of $T^*$.  In each step that we traverse $e$ on the dual edge $e^*$ from $v_1^*$ and $v_2^*$, assign the crossing \diaCircle generated by $e$ to the corner that belongs to the regular region dual to $v_2^*$.

\end{enumerate}

The above construction may be easier to understand if one looks instead at the more concrete pictures in Example \ref{Ex1} below.  Since every vertex $v\neq r$ of an oriented spanning tree $T$ has in-degree 1, and vertices in $T^*$ have a one-to-one correspondence with regular regions of $\mathbb{R}^2$ separated by $G$,  one can see that $s$ thus defined gives a bijective map between $\operatorname{Cr}(G)$ and $\operatorname{Re}(G)\backslash \{R_u, R_{v}\}$; so it is a Kauffman state by Definition \ref{states}.  Therefore, the map $\phi: \cal{T}_r(G)   \longrightarrow   S(G, \delta) $ is well-defined.

To show that $\phi$ is bijective, we construct an inverse map $\psi: S(G, \delta) \longrightarrow \cal{T}_r(G)$.  Given a Kauffman state $s$, let $F\subset E$ be the set of edges so that $s$ assigns the crossing \diaCircle generated by those edges to their north corners. By definition $e_1 \in F$, recalling that $e_1$ is the edge with base point $\delta$. Then $E-F$ is the set of edges so that $s$ assigns the crossing \diaCircle generated by those edges to their east or west corners.  Let $T$ be the subgraph of $G$ generated by $F-\{e_1\}$, and let $T^*$ be the subgraph of $G^*$ generated by $(E-F)\cup \{e_1\}$. We want to show that $T$ is an oriented spanning tree rooted at $r$.  To this end, note that the size of $T$ is by definition $|F|-1=|V|-1$ since $s$ is a Kauffman state.  It is also clear that every vertex $v\neq r$ has in-degree $1$ and the root $r$ has in-degree $0$.  Thus, it suffices to show that $T$ does not have a cycle.

We prove by contradiction.  Suppose $C$ is a cycle in $T$.  Then $C$ bounds a disk $D$ in $\mathbb{R}^2$.  Without loss of generality, we assume that $D\cap  \mathrm{Int}(e_1) =\emptyset$, and therefore the marked regions $R_u, R_v$ are not contained in $D$.  Consider the subgraph $G'=G\cap D$.  Let $a$ be the number of vertices of $G'$, and let $b$ be the number of edges of $G'$.  By Euler's formula, the number of regular regions of $G$ inside $D$  is $b-a+1$.  
Together with the additional $a$ circle regions intersecting $D$, the total number is $$ \#( \text{regions})= (b-a+1)+a=b+1.$$
Meanwhile, the total number of crossings in $D$ is  $$ \#( \text{crossings}  \diaCircle )= b.$$
Since $\partial D=C\subset T$, the Kauffman state $s$ assigns the crossing \diaCircle generated by edges of the cycle $C$ to their north corners (circle regions intersecting the boundary of $D$).  It follows that $s$ must map $b$ crossings in $D$ onto $b+1$ regions in $D$, which is impossible.

Thus, we proved $T$ is a spanning tree, and we define $\psi(s)=T$. Clearly, $\psi$ is the inverse of $\phi$.  It is straightforward to see that the weight of each spanning tree $c(T)$ is equal to the corresponding term $P_s(1)$.  This proves the theorem.
\end{proof}

\begin{example}\label{Ex1}

The graph in Example \ref{ExampleSpanningTree} is in fact an MOY graph diagram, so we can compute its Alexander polynomial.   With the base point $\delta$ on the edge of weight $k$, we obtain a decorated diagram and find exactly one Kauffman state $s$, as indicated by $\bullet$ in Fig.  \ref{fig:example} (left).  The associated spanning tree $T$ rooted at $v_3$ and its dual spanning tree $T^*$ specified by Theorem \ref{Correspondence} are marked in thick red in Fig.  \ref{fig:example} (right).

\begin{figure}[h!]
\begin{tikzpicture}[baseline=-0.65ex, thick, scale=1.5]
\draw (0, 1.5) node[circle,fill,inner sep=1pt]{};
\draw (1, -0.5) node[circle,fill,inner sep=1pt]{};
\draw (-1, -0.5) node[circle,fill,inner sep=1pt]{};

\draw (1, -0.5) [->-] to (0, 1.5);
\draw (-1, -0.5) [->-] to (0, 1.5);
\draw (-1, -0.5) [->-] to (1, -0.5);

\draw (0,1.5)  [->-] to [out=30,in=45] (1,-0.5);
\draw (0,1.5)  [->-] to [out=150,in=135] (-1,-0.5);
\draw (0, 1.7) node {$v_1$};
\draw (-1.2, -0.6) node {$v_2$};
\draw (1.2, -0.6) node {$v_3$};

\draw (0, -0.8) node {$k$};
\draw (0.1, 0.5) node {$i+k$};
\draw (-0.7, 0.5) node {$j$};
\draw (1.3, 0.7) node {$i$};
\draw (-1.5, 0.7) node {$j+k$};

\draw (0, 1.5) circle (0.4);
\draw (1, -0.5) circle (0.4);
\draw (-1, -0.5) circle (0.4);

\draw(0.2, -0.5) node{$*$};
\draw(0.3, -0.3) node{$\delta$};

\draw(0.8, -0.5) node {$\bullet$}; 
\draw(0.1, 1.3) node {$\bullet$};
\draw(-1.15, -0.3) node {$\bullet$};
\draw(-0.3, 1.1) node {$\bullet$};
\draw(1.1, 0) node {$\bullet$};
\end{tikzpicture} \quad \quad \quad
\begin{tikzpicture}[baseline=-0.65ex, thick, scale=1.4]
\draw (0, 2) node[circle,fill,inner sep=1.5pt]{};
\draw (1, 0) node[circle,fill,inner sep=1.5pt]{};
\draw (-1, 0) node[circle,fill,inner sep=1.5pt]{};

\draw (1, 0) [->-, color=red, line width=0.7mm] to (0, 2);
\draw (-1, 0) [->-] to (0, 2);
\draw (-1, 0) [->-] to (1, 0);

\draw (0,2)  [->-] to [out=30,in=45] (1,0);
\draw (0,2)  [->-, color=red, line width=0.7mm] to [out=150,in=135] (-1,0);
\draw (0, 2.2) node {$v_1$};
\draw (-1.2, -0.1) node {$v_2$};
\draw (1.2, -0.1) node {$v_3$};

\draw (-0.8, 2) [color=red, line width=0.7mm] node {$T$};

\draw (0, 1) node{$\circ$};
\draw (0, -1) node{$\circ$};
\draw (0.7, 1.35) node{$\circ$};
\draw (-0.7, 1.35) node{$\circ$};

\draw [dashed, color=red, line width=0.7mm](0.7,1.35) to [out=0,in=90] (1.8,0) to [out=270, in=0] (0,-1);

\draw [dashed](-0.7,1.35) to [out=180,in=90] (-1.8,0) to [out=270, in=180] (0,-1);

\draw[color=red, dashed, line width=0.7mm] (0, 1) --  (0, -1);
\draw[color=red, dashed, line width=0.7mm] (0, 1) --  (-0.7, 1.35);
\draw[dashed] (0, 1) --  (0.7, 1.35);

\draw (0.3, -0.5) [color=red] node {$T^*$};
\end{tikzpicture}

\caption{The Kauffman state $s\in S(G, \delta)$ in the decorated diagram (left) and the corresponding oriented spanning tree $T$ and its dual spanning tree $T^*$ in thick red (right).}

\label{fig:example}
\end{figure}
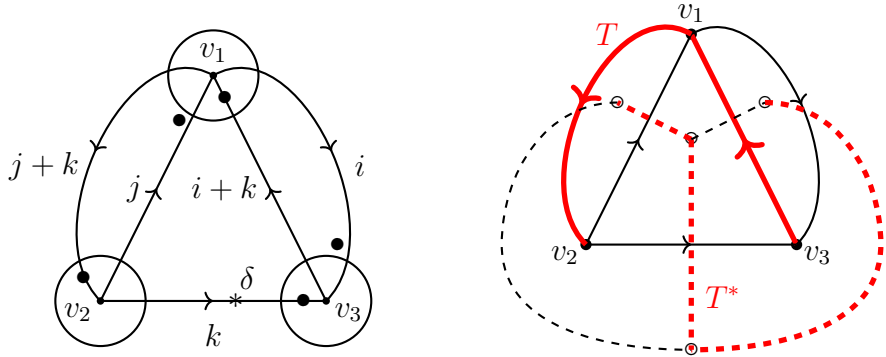

According to Definition \ref{Def:KauffmanStateSum}, $$\Delta_{(G,c)}(t)=P_s(t)= t^{k/2}\cdot t^{-j/2}\cdot t^{i/2}\cdot [i+k]\cdot [j+k],   $$
as $s$ is the unique Kauffman state.   In particular, we see that $$P_s(1)=(i+k)(j+k)=c(T).$$ 

\end{example}


\section{Spanning trees and skein relations}

To establish Theorem \ref{MainTheorem} for arbitrary graphs, our strategy is to prove a set of skein relations and reduce the general case to the plane graph case, which was  just confirmed in the previous section.

Let $(g, c)$ be an abstract MOY graph, and let $G$ be an MOY graph diagram of $g$ on $\mathbb{R}^2$.  In general, $G$ may have double points corresponding to crossings of type \diaCrossP (positive crossing) and \diaCrossN (negative crossing).  Since neither of the invariant $\Delta_{(g, c)}$ or $N(g, c)$ depends on the types of crossings, hereafter, we simply use \diaCross to represent a double point (for either positive or negative crossing) in $G$. We begin with a lemma.

\begin{lemma}\label{insertvertex}
Let $(g, c)$ be an abstract MOY graph. We obtain a new graph $(g', c')$ by inserting a vertex $v'$ of degree 2 into an edge $e$ of $g$. Then we have $$N(g', c')=c(e)N(g, c),$$ where $c'$ denotes the induced balanced weight on $g'$ from $c$.
\end{lemma}
\begin{proof}
Suppose $e$ in $g$ is separated into two edges $e_1$ and $e_2$ in $g'$, and $e_1$ is the edge pointing to $v'$.
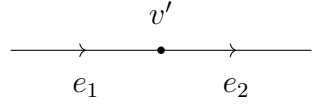
\begin{align*}
\begin{tikzpicture}
\draw (-1, 0) [->-] to (1, 0);
\draw (1,0.5) node {$v'$};
\draw (1, 0) [->-] to (3, 0);
\draw (1,0) node[circle,fill,inner sep=1pt]{};
\draw (0,-0.5) node {$e_1$};
\draw (2,-0.5) node {$e_2$};
\end{tikzpicture}
\end{align*}
For any root vertex $r$ in $g$, consider $\mathcal {T}_r(g)$ and $\mathcal {T}_{r}(g')$, the set of all oriented spanning trees of $g$ and $g'$ rooted at $v$, respectively. There is a canonical one-to-one correspondence between $\mathcal {T}_r(g)$ and $\mathcal {T}_{r}(g')$: if $T\in \mathcal {T}_r(g)$ contains $e$, let $T'=(T-\{e\})\cup \{e_1, e_2\}$; if $T \in \mathcal {T}_r(g)$ does not contain $e$, let  $T'=T\cup \{e_1\}$. In either case, we have $c'(T')=c(e)c(T)$.  Taking the sum over all trees gives the lemma.
\end{proof}

\begin{prop}\label{skeinrelation}
We have the following skein relations for the weighted number of spanning trees, where $N(G)$ represents $N(g, c)$ if $G$ is a graph diagram with underlying graph $(g, c)$.
In each equality, the graph diagrams are identical outside the local diagrams shown there. When $i=j$, ignore the edge with weight $j-i$.
\begin{align*}
&\text{When $i\leq j$:}\\
&N\left(\begin{tikzpicture}[baseline=-0.65ex, thick, scale=0.6]
\draw (-1, -1) [->] -- (1, 1) node[above]{$i$};
\draw (1,-1) [->] to (-1,1)  node[above]{$j$};
\end{tikzpicture}\right)
  =\frac{-1}{i\cdot j}\cdot
N\left(\begin{tikzpicture}[baseline=-0.65ex, thick, scale=1.2]
\draw (0,-1) [->-] to (0, 0.33);
\draw (0, 0.33) [->] to (0,1);
\draw (1,-1) [->-] to (1, -0.33);
\draw (1, -0.33) [->] to (1,1);
\draw (0,0.33) [-<-]  to [out=270,in=180] (0.5,0) to [out=0,in=90] (1,-0.33);
\draw (0,1.25) node {$j$};
\draw (0,-1.25) node {$i$};
\draw (1,1.25) node {$i$};
\draw (0.9,-1.25) node {$j$};
\draw (0.5,0.3) node {$j-i$};
\draw (1, -0.33) node[circle,fill,inner sep=1pt]{};
\draw (0, 0.33) node[circle,fill,inner sep=1pt]{};
\end{tikzpicture}\right)
+ \, \frac{1}{i\cdot (i+j)}\cdot
N\left(
\begin{tikzpicture}[baseline=-0.65ex, thick, scale=1.2]
\draw (0,-1) [->-] to [out=90,in=270] (0.5,-0.33);
\draw (0.5, -0.33) [->-] to [out=90,in=270] (0.5,0.33);
\draw (0.5, 0.33) [->] to [out=90,in=270] (0,1);
\draw (1,-1) [->-] to [out=90,in=270] (0.5,-0.33);
\draw (1,1) [<-] to [out=270,in=90] (0.5,0.33);
\draw (0, -1.25) node {$i$};
\draw (0.9,-1.25) node {$j$};
\draw (0,1.25) node {$j$};
\draw (1,1.25) node {$i$};
\draw (1,0) node {$i+j$};
\draw (0.5, -0.33) node[circle,fill,inner sep=1pt]{};
\draw (0.5, 0.33) node[circle,fill,inner sep=1pt]{};
\end{tikzpicture}\right).\\
 &\text{When $j< i$:}\\
&\left(\begin{tikzpicture}[baseline=-0.65ex, thick, scale=0.6]
\draw (-1, -1) [->] -- (1, 1) node[above]{$i$};
\draw  (1,-1) [->] to (-1,1)  node[above]{$j$};
\end{tikzpicture}\right)
  =\frac{-1}{i\cdot j}\cdot
\left(\begin{tikzpicture}[baseline=-0.65ex, thick, scale=1.2]
\draw (0,-1) [->-] to (0, -0.33);
\draw (0, -0.33) [->] to (0,1);
\draw (1,-1) [->-] to (1, 0.33);
\draw (1, 0.33) [->] to (1,1);
\draw (0,-0.33) [->-]  to [out=90,in=180] (0.5,0) to [out=0,in=270] (1,0.33);
\draw (0,1.25) node {$j$};
\draw (0,-1.25) node {$i$};
\draw (1,1.25) node {$i$};
\draw (0.9,-1.25) node {$j$};
\draw (0.5,0.3) node {$i-j$};
\draw (1, 0.33) node[circle,fill,inner sep=1pt]{};
\draw (0, -0.33) node[circle,fill,inner sep=1pt]{};
\end{tikzpicture}\right)
+ \, \frac{1}{j\cdot (i+j)}\cdot
\left(
\begin{tikzpicture}[baseline=-0.65ex, thick, scale=1.2]
\draw (0,-1) [->-] to [out=90,in=270] (0.5,-0.33);
\draw (0.5, -0.33) [->-] to [out=90,in=270] (0.5,0.33);
\draw (0.5, 0.33) [->] to [out=90,in=270] (0,1);
\draw (1,-1) [->-] to [out=90,in=270] (0.5,-0.33);
\draw (1,1) [<-] to [out=270,in=90] (0.5,0.33);
\draw (0, -1.25) node {$i$};
\draw (0.9,-1.25) node {$j$};
\draw (0,1.25) node {$j$};
\draw (1,1.25) node {$i$};
\draw (1,0) node {$i+j$};
\draw (0.5, -0.33) node[circle,fill,inner sep=1pt]{};
\draw (0.5, 0.33) node[circle,fill,inner sep=1pt]{};
\end{tikzpicture}\right).\\
\end{align*}
\end{prop}

\begin{proof}
We prove the first relation only and the second one can be proved analogously. Denote $G, G_1, G_2$ the diagram on the left hand side and the two diagram on the right hand side of the equality, respectively.

We assume without loss of generality that there are four vertices $a, b, c, d$ of degree $2$ in each of the diagram as shown below. If not, we will just insert the missing ones: Lemma \ref{insertvertex} ensures that the invariants $N(G)$, $N(G_1)$ and $N(G_2)$ will change by a same factor.

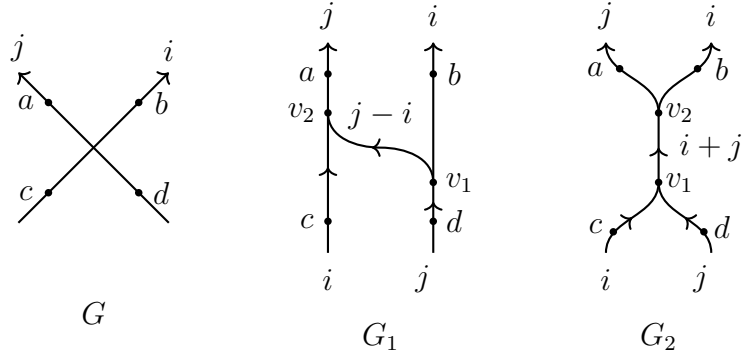
\begin{align*}
\begin{tikzpicture}[baseline=-0.65ex, thick, scale=1]
\draw (-1, -1) [->] -- (1, 1) node[above]{$i$};
\draw (1,-1) [->] to (-1,1)  node[above]{$j$};
\draw (0.6,0.6) node[circle,fill,inner sep=1pt]{};
\draw (-0.6,0.6) node[circle,fill,inner sep=1pt]{};
\draw (0.6,-0.6) node[circle,fill,inner sep=1pt]{};
\draw (-0.6,-0.6) node[circle,fill,inner sep=1pt]{};
\draw (0.9,0.6)  node{$b$};
\draw (-0.9,0.6)  node{$a$};
\draw (-0.9,-0.6)  node{$c$};
\draw (0.9,-0.6)  node{$d$};
\draw (0,-2.2)  node{$G$};
\end{tikzpicture}\quad\quad\quad
\begin{tikzpicture}[baseline=-0.65ex, thick, scale=1.4]
\draw (0,-0.7) [->-] to (0, 0.33);
\draw (0,-1) -- (0, -0.7);
\draw (0, 0.33) [->] to (0,1);
\draw (1,-0.7) [->-] to (1, -0.33);
\draw (1,-1) -- (1, -0.7);
\draw (1, -0.33) [->] to (1,1);
\draw (0,0.33) [-<-]  to [out=270,in=180] (0.5,0) to [out=0,in=90] (1,-0.33);
\draw (0,1.25) node {$j$};
\draw (0,-1.25) node {$i$};
\draw (1,1.25) node {$i$};
\draw (0.9,-1.25) node {$j$};
\draw (0.5,0.3) node {$j-i$};
\draw (1, -0.33) node[circle,fill,inner sep=1pt]{};
\draw (0, 0.33) node[circle,fill,inner sep=1pt]{};
\draw (0,-0.7) node[circle,fill,inner sep=1pt]{};
\draw (-0.2,-0.7) node {$c$};
\draw (1,-0.7) node[circle,fill,inner sep=1pt]{};
\draw (1.2,-0.7) node {$d$};
\draw (0,0.7) node[circle,fill,inner sep=1pt]{};
\draw (-0.2,0.7) node {$a$};
\draw (1,0.7) node[circle,fill,inner sep=1pt]{};
\draw (1.2,0.7) node {$b$};
\draw (0.5,-1.8) node {$G_1$};
\draw (1.25,-0.33) node {$v_1$};
\draw (-0.25,0.33) node {$v_2$};
\end{tikzpicture}
 \quad\quad\quad
\begin{tikzpicture}[baseline=-0.65ex, thick, scale=1.4]
\draw (0,-1) [->-] to [out=90,in=270] (0.5,-0.33);
\draw (0.5, -0.33) [->-] to [out=90,in=270] (0.5,0.33);
\draw (0.5, 0.33) [->] to [out=90,in=270] (0,1);
\draw (1,-1) [->-] to [out=90,in=270] (0.5,-0.33);
\draw (1,1) [<-] to [out=270,in=90] (0.5,0.33);
\draw (0, -1.25) node {$i$};
\draw (0.9,-1.25) node {$j$};
\draw (0,1.25) node {$j$};
\draw (1,1.25) node {$i$};
\draw (1,0) node {$i+j$};
\draw (0.5, -0.33) node[circle,fill,inner sep=1pt]{};
\draw (0.5, 0.33) node[circle,fill,inner sep=1pt]{};
\draw (0.5,-1.8) node {$G_2$};
\draw (0.93,-0.8) node[circle,fill,inner sep=1pt]{};
\draw (0.13,0.75) node[circle,fill,inner sep=1pt]{};
\draw (0.87,0.75) node[circle,fill,inner sep=1pt]{};
\draw (0.07,-0.8) node[circle,fill,inner sep=1pt]{};
\draw (1.1,0.75) node{$b$};
\draw (-0.1,0.75) node{$a$};
\draw (-0.1,-0.75) node{$c$};
\draw (1.1,-0.75) node{$d$};
\draw (0.7,-0.33) node{$v_1$};
\draw (0.7,0.33) node{$v_2$};
\end{tikzpicture}
\end{align*}

Proposition \ref{BalancedWeightCount}, which claims that the weighted number of spanning trees is independent of the choice of the root for a balanced weight, enables us to further simplify our argument. In each of $G, G_1, G_2$, we choose $b$ to be the root and analyze the shape of the corresponding spanning trees. 

An oriented spanning tree of $G$ rooted at $b$ must contain the edge pointing to $c$, the edge pointing to $d$, the edge pointing out of $b$ and the edge $da$, which are highlighted in thick red as below.
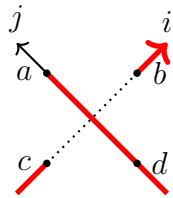
\begin{align*}
\begin{tikzpicture}[baseline=-0.65ex, thick, scale=1][h!]
\draw [line width=0.7mm, red] (0.6,0.6) [->] -- (1, 1);
\draw (1,1) node[above]{$i$};
\draw [dotted] (-0.6,-0.6) -- (0.6,0.6);
\draw [line width=0.7mm, red] (-1, -1) -- (-0.6,-0.6);
\draw (-0.6,0.6) [->] to (-1,1)  node[above]{$j$};
\draw [line width=0.7mm, red] (1, -1) -- (-0.6,0.6);
\draw (0.6,0.6) node[circle,fill,inner sep=1pt]{};
\draw (-0.6,0.6) node[circle,fill,inner sep=1pt]{};
\draw (0.6,-0.6) node[circle,fill,inner sep=1pt]{};
\draw (-0.6,-0.6) node[circle,fill,inner sep=1pt]{};
\draw (0.9,0.6)  node{$b$};
\draw (-0.9,0.6)  node{$a$};
\draw (-0.9,-0.6)  node{$c$};
\draw (0.9,-0.6)  node{$d$};
\end{tikzpicture}
\end{align*}

An oriented spanning tree of $G_1$ rooted at $b$
must contain the edge pointing to $c$, the edge pointing to $d$, the edge $dv_1$, the edge $v_2a$, the edge pointing out of $b$ and either one of the edges in the following 2 cases:
\begin{itemize}
\item [(A)] the edge $v_1v_2$
\item [(B)] the edge $cv_2$,
\end{itemize}
which are highlighted in thick red as below.

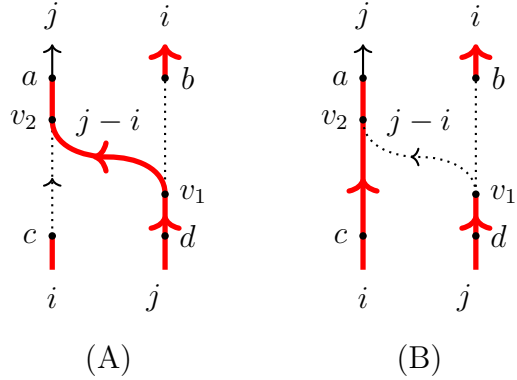
\begin{align*}
\begin{tikzpicture}[baseline=-0.65ex, thick, scale=1.5][h!]
\draw [dotted] (0,-0.7) [->-] to (0, 0.33);
\draw [line width=0.7mm, red] (0,-1) -- (0, -0.7);
\draw [line width=0.7mm, red] (0, 0.33) -- (0,0.7);
\draw (0, 0.7) [->] to (0,1);
\draw [line width=0.7mm, red] (1,-0.7) [->-] to (1, -0.33);
\draw [line width=0.7mm, red] (1,-1) -- (1, -0.7);
\draw [dotted](1, -0.33) -- (1,0.7);
\draw [line width=0.7mm, red] (1, 0.7) [->] to (1,1);
\draw [line width=0.7mm, red] (0,0.33) [-<-]  to [out=270,in=180] (0.5,0) to [out=0,in=90] (1,-0.33);
\draw (0,1.25) node {$j$};
\draw (0,-1.25) node {$i$};
\draw (1,1.25) node {$i$};
\draw (0.9,-1.25) node {$j$};
\draw (0.5,0.3) node {$j-i$};
\draw (1, -0.33) node[circle,fill,inner sep=1pt]{};
\draw (0, 0.33) node[circle,fill,inner sep=1pt]{};
\draw (0,-0.7) node[circle,fill,inner sep=1pt]{};
\draw (-0.2,-0.7) node {$c$};
\draw (1,-0.7) node[circle,fill,inner sep=1pt]{};
\draw (1.2,-0.7) node {$d$};
\draw (0,0.7) node[circle,fill,inner sep=1pt]{};
\draw (-0.2,0.7) node {$a$};
\draw (1,0.7) node[circle,fill,inner sep=1pt]{};
\draw (1.2,0.7) node {$b$};
\draw (0.5,-1.8) node {(A)};
\draw (1.25,-0.33) node {$v_1$};
\draw (-0.25,0.33) node {$v_2$};
\end{tikzpicture}\quad\quad\quad
\begin{tikzpicture}[baseline=-0.65ex, thick, scale=1.5]
\draw [line width=0.7mm, red] (0,-0.7) [->-] to (0, 0.33);
\draw [line width=0.7mm, red] (0,-1) -- (0, -0.7);
\draw [line width=0.7mm, red] (0, 0.33) -- (0,0.7);
\draw (0, 0.7) [->] to (0,1);
\draw [line width=0.7mm, red] (1,-0.7) [->-] to (1, -0.33);
\draw [line width=0.7mm, red] (1,-1) -- (1, -0.7);
\draw [dotted](1, -0.33) -- (1,0.7);
\draw [line width=0.7mm, red] (1, 0.7) [->] to (1,1);
\draw [dotted] (0,0.33) [-<-]  to [out=270,in=180] (0.5,0) to [out=0,in=90] (1,-0.33);
\draw (0,1.25) node {$j$};
\draw (0,-1.25) node {$i$};
\draw (1,1.25) node {$i$};
\draw (0.9,-1.25) node {$j$};
\draw (0.5,0.3) node {$j-i$};
\draw (1, -0.33) node[circle,fill,inner sep=1pt]{};
\draw (0, 0.33) node[circle,fill,inner sep=1pt]{};
\draw (0,-0.7) node[circle,fill,inner sep=1pt]{};
\draw (-0.2,-0.7) node {$c$};
\draw (1,-0.7) node[circle,fill,inner sep=1pt]{};
\draw (1.2,-0.7) node {$d$};
\draw (0,0.7) node[circle,fill,inner sep=1pt]{};
\draw (-0.2,0.7) node {$a$};
\draw (1,0.7) node[circle,fill,inner sep=1pt]{};
\draw (1.2,0.7) node {$b$};
\draw (0.5,-1.8) node {(B)};
\draw (1.25,-0.33) node {$v_1$};
\draw (-0.25,0.33) node {$v_2$};
\end{tikzpicture}
\end{align*}

An oriented spanning tree of $G_2$ rooted at $b$
must contain the edge pointing to $c$, the edge pointing to $d$, the edge $v_1v_2$, the edge $v_2a$, the edge pointing out of $b$ and either one of the edges in the following 2 cases:
\begin{itemize}
\item [($\alpha$)] the edge $dv_1$,
\item [($\beta$)] the edge $cv_1$,
\end{itemize}
which are highlighted in thick red as below.

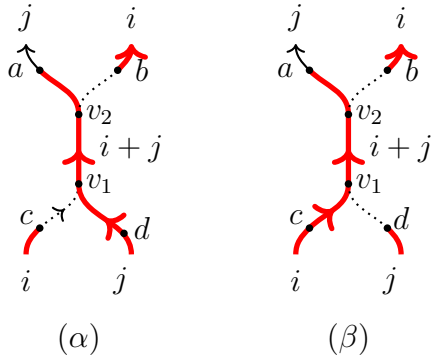
\begin{align*}
\begin{tikzpicture}[baseline=-0.65ex, thick, scale=1.4][h!]
\draw [line width=0.7mm, red] (0,-1) to [out=90,in=225] (0.13,-0.75);
\draw [dotted] (0.13,-0.75) [->-] to [out=45,in=270] (0.5,-0.33);
\draw [line width=0.7mm, red] (0.5, -0.33) [->-] to [out=90,in=270] (0.5,0.33);
\draw [line width=0.7mm, red] (0.5, 0.33)  to [out=90,in=315] (0.13,0.75);
\draw  (0.13,0.75) [->] to [out=135,in=270] (0,1);
\draw [line width=0.7mm, red] (1,-1) [->-] to [out=90,in=270] (0.5,-0.33);
\draw [line width=0.7mm, red] (0.87,0.75) [->] to [out=45,in=270] (1, 1);
\draw [dotted] (0.5, 0.33)  to [out=90,in=225] (0.87,0.75);
\draw (0, -1.25) node {$i$};
\draw (0.9,-1.25) node {$j$};
\draw (0,1.25) node {$j$};
\draw (1,1.25) node {$i$};
\draw (1,0) node {$i+j$};
\draw (0.5, -0.33) node[circle,fill,inner sep=1pt]{};
\draw (0.5, 0.33) node[circle,fill,inner sep=1pt]{};
\draw (0.5,-1.8) node {$(\alpha)$};
\draw (0.93,-0.8) node[circle,fill,inner sep=1pt]{};
\draw (0.13,0.75) node[circle,fill,inner sep=1pt]{};
\draw (0.87,0.75) node[circle,fill,inner sep=1pt]{};
\draw (0.13,-0.75) node[circle,fill,inner sep=1pt]{};
\draw (1.1,0.75) node{$b$};
\draw (-0.1,0.75) node{$a$};
\draw (0,-0.65) node{$c$};
\draw (1.1,-0.75) node{$d$};
\draw (0.7,-0.33) node{$v_1$};
\draw (0.7,0.33) node{$v_2$};
\end{tikzpicture}\quad\quad\quad
\begin{tikzpicture}[baseline=-0.65ex, thick, scale=1.4]
\draw [line width=0.7mm, red] (0,-1) to [out=90,in=225] (0.13,-0.75);
\draw [line width=0.7mm, red] (0.13,-0.75) [->-] to [out=45,in=270] (0.5,-0.33);
\draw [line width=0.7mm, red] (0.5, -0.33) [->-] to [out=90,in=270] (0.5,0.33);
\draw [line width=0.7mm, red] (0.5, 0.33)  to [out=90,in=315] (0.13,0.75);
\draw  (0.13,0.75) [->] to [out=135,in=270] (0,1);
\draw [line width=0.7mm, red] (1,-1) to [out=90,in=315] (0.87,-0.75);
\draw [dotted] (0.87,-0.75) to [out=135,in=270] (0.5, -0.33);
\draw [line width=0.7mm, red] (0.87,0.75) [->] to [out=45,in=270] (1, 1);
\draw [dotted] (0.5, 0.33)  to [out=90,in=225] (0.87,0.75);
\draw (0, -1.25) node {$i$};
\draw (0.9,-1.25) node {$j$};
\draw (0,1.25) node {$j$};
\draw (1,1.25) node {$i$};
\draw (1,0) node {$i+j$};
\draw (0.5, -0.33) node[circle,fill,inner sep=1pt]{};
\draw (0.5, 0.33) node[circle,fill,inner sep=1pt]{};
\draw (0.5,-1.8) node {$(\beta)$};
\draw (0.87,-0.75) node[circle,fill,inner sep=1pt]{};
\draw (0.13,0.75) node[circle,fill,inner sep=1pt]{};
\draw (0.87,0.75) node[circle,fill,inner sep=1pt]{};
\draw (0.13,-0.75) node[circle,fill,inner sep=1pt]{};
\draw (1.1,0.75) node{$b$};
\draw (-0.1,0.75) node{$a$};
\draw (0,-0.65) node{$c$};
\draw (1,-0.65) node{$d$};
\draw (0.7,-0.33) node{$v_1$};
\draw (0.7,0.33) node{$v_2$};
\end{tikzpicture}
\end{align*}

Note that an oriented spanning tree $T$ of $G$ rooted at $b$ corresponds to a unique tree $T_1$ of type (A) in $G_1$ and a unique tree $T_2$ of type ($\alpha$) in $G_2$, and vice versa. Hence,
$$\mathcal{T}_b(G) \xleftrightarrow{1:1} \{ \text{type (A) in } \mathcal{T}_b(G_1) \} \xleftrightarrow{1:1} \{ \text{type ($\alpha$) in } \mathcal{T}_b(G_2) \} .$$
Under this correspondence, we can check that
$$c(T)=\frac{-1}{i\cdot j} c(T_1)+\frac{1}{i \cdot (i+j)}c(T_2).$$

Similarly, an oriented spanning tree $T'_1$ of type (B) in $G_1$ corresponds to a unique tree $T'_2$ of type ($\beta$) in $G_2$, and vice versa.  Hence,
$$\{ \text{type (B) in } \mathcal{T}_b(G_1) \} \xleftrightarrow{1:1} \{ \text{type ($\beta$) in } \mathcal{T}_b(G_2) \}.$$
Under this correspondence, we have
$$\frac{-1}{i\cdot j} c(T'_1)+\frac{1}{i \cdot (i+j)}c(T'_2)=0.$$

Finally, we sum up over all trees and apply the above two identities on their weights to obtain the desired equality.  This completes the proof.
\end{proof}

Now we are ready to prove our main theorem.

\begin{proof}[Proof of Theorem \ref{MainTheorem}]
The key observation is that the skein relations in Proposition \ref{skeinrelation} for $N(g,c)$ is the same as the ones for $\Delta_{(g, c)}=\Delta_{(G, c)}(1)$, obtained by substituting $t=1$ in \cite[Theorem 4.1 (iv)]{BW}. Note that Theorem \ref{MainTheorem} for plane MOY graphs has been proved in Theorem \ref{Correspondence}, and a general MOY graph diagram can be related to plane graphs by a finite number of skein relations.  It follows by induction that Theorem \ref{MainTheorem} holds for arbitrary MOY graphs.
\end{proof}

We conclude this section by proving Corollary \ref{Nonvanishing}.  This follows directly from the following two lemmas, since existence of spanning trees implies positivity of weighted number of spanning trees when the weight function is positive.

\begin{lemma} A connected directed graph with a balanced positive weight is strongly connected, i.e., every vertex is reachable from every other vertex by a directed path.
\end{lemma}

\begin{proof}
Suppose $\Gamma$ is a connected graph with a positive balanced weight $c$.  Given a vertex $v\in V$, let $S \subset V$ be the set of all vertices that can be reached from $v$. If $S$ is a proper subset of $V$, then $V-S$ is not empty.  As $\Gamma$ is a connected graph, there must be edges that connect vertices in $S$ with vertices in $V-S$. Let $F$ be the set of such edges.  Applying the positivity and the balanced condition (\ref{balancedcoloring}) of $c$ on all vertices in $S$, we can further see that there must be some edges in $F$ that are oriented from some vertices in $S$ to some vertices in $V-S$.  Then there is a vertex in $V-S$ which is also reachable from $v$ by a directed path.  This contradicts the definition of $S$, so we must have $S=V$.  

\end{proof}

\begin{lemma}
Every strongly connected graph has an oriented spanning tree with any given root.
\end{lemma}

\begin{proof}
This is a standard result in graph theory.  For any given root $r$, simply take a maximal oriented tree rooted at $r$.  Such a tree must be spanning by the strongly connected assumption.
\end{proof}

\end{document}